\documentclass[reqno,a4paper]{amsart}
\usepackage{amssymb}
\usepackage{amsmath}
\usepackage{amsfonts}

\newtheorem{theorem}{Theorem}[section]
\theoremstyle{plain}

\newtheorem{corollary}{Corollary}[section]

\newtheorem{definition}{Definition}[section]

\newtheorem{lemma}{Lemma}[section]

\numberwithin{equation}{section}
\input{tcilatex}

\begin{document}
\title[A new approach on helices in pseudo-Riemannian manifolds]{A new
approach on helices in pseudo-Riemannian manifolds}
\urladdr{}
\author{Evren Z\i plar}
\address{Department of Mathematics, Faculty of Science, Cank\i r\i \
University, Cank\i r\i , Turkey}
\email{evrenziplar@karatekin.edu.tr}
\author{Yusuf Yayl\i }
\address{Department of Mathematics, Faculty of Science, University of
Ankara, Tando\u{g}an, Turkey}
\email{yayli@science.ankara.edu.tr}
\author{\.{I}smail G\"{O}K}
\address{Department of Mathematics, Faculty of Science, University of
Ankara, Tando\u{g}an, Turkey}
\email{igok@science.ankara.edu.tr}
\urladdr{}
\date{September 9, 2013. \ AMS subject classifications:\textbf{\ }14H45,
14H50, 53C50}

\begin{abstract}
In this paper, we give a definition of harmonic curvature functions in terms
of $V_{n}$ and define a new kind of slant helix which is called $V_{n}-$%
slant helix in $n-$dimensional pseudo-Riemannian manifold. Also, we give
important characterizations about the helix. \  \  \  \  \  \  \  \  \  \  \  \  \  \  \  \
\  \  \  \  \  \  \  \  \  \  \  \  \  \  \  \  \  \  \  \  \  \  \  \  \  \  \  \  \  \  \  \  \  \  \  \  \  \
\  \  \  \  \  \  \  \  \  \  \  \  \  \  \ 
\end{abstract}

\maketitle

\section{Introduction}

Curves theory is an important framework in the differential geometry
studies. Helix is one of the most fascinating curves because we see helical
structure in nature, science and mechanical tools. Helices arise in the
field of computer aided design, computer graphics, the simulation of
kinematic motion or design of highways, the shape of DNA and carbon
nonotubes. Also, we can see the helical structure in fractal geometry, for
instance hyperhelices \cite{Jain, Yin}.

Furthermore, helices share common origins in the geometries of the platonic
solids, with inherent hierarchical potential that is typical of biological
structures. The helices provide an energy-efficient solution to
close-packing in molecular biology, a common motif in protein construction,
and a readily observable pattern at many size levels throughout the body.
The helices are described in a variety of anatomical structures, suggesting
their importance to structural biology and manual therapy \cite{Scarr}.

A curve of constant slope or general helix in Euclidean 3-space $E^{3}$ is
defined by the property that its tangent vector field makes a constant angle
with a fixed straight line (the axis of general helix). A classical result
stated by Lancret in 1802 and first proved by de Saint Venant in 1845 (\cite%
{Lancret} and \cite{Struik}) is: A necessary and sufficient condition that a
curve be a general helix is that the ratio of curvature to torsion be
constant. In \cite{Haci}, \"{O}zdamar and Hac\i saliho\u{g}lu defined
harmonic curvature functions $H_{i}$ $\left( 1\leq i\leq n-2\right) $ of a
curve $\alpha $ and generalized helices in $E^{3}$ to in $n-$dimensional
Euclidean space $E^{n}$. \ Moreover, they gave a characterization for the
inclined curves in $E^{n}$ : 
\begin{equation}
\text{\textquotedblleft A curve is an inclined curve if and only if }\dsum
\limits_{i=1}^{n}H_{i}^{2}=\text{constant\textquotedblright }
\end{equation}

Harmonic curvature functions have important role in characterizations of
general helices in higher dimensions. Because, the notion of a general helix
can be generalized to higher dimension in different ways. However, these
ways are not easy to show which curves are general helices and finding the
axis of a general helix is complicated in higher dimension. Thanks to
harmonic curvature functions, we can easily obtain the axis of such curves.
Moreover, this way is confirmed in 3-dimensional spaces.

Izumiya and Takeuchi defined a new kind of helix (slant helix) and they gave
a characterization of slant helices in Euclidean $3-$space $E^{3}$ \cite%
{izumiya}. In 2008, \"{O}nder \emph{et al}. defined a new kind of slant
helix in Euclidean $4-$space $E^{4}$ which is called $B_{2}-$slant helix and
they gave some characterizations of these slant helices in Euclidean $4-$%
space $E^{4}$ \cite{onder} . And then in 2009, G\"{o}k \emph{et al}.
generalized $B_{2}-$slant helix in $E^{4}$ to $E^{n}$, $n>3$, called $V_{n}-$%
slant helix in Euclidean and Minkowski $n$-space (\cite{gok, gok1}). Lots of
authors in their papers have investigated inclined curves and slant helices
using the harmonic curvature functions in Euclidean and Minkowski $n$-space (%
\cite{ali, ali1, klah, oz, cam}). But, Z\i plar et al.(\cite{zip}) see for
the first time that the characterization of inclined curves and slant
helices in $(1.1)$ is true only for the case necessity but not true for the
case sufficiency in Euclian $n$-space. Then, they consider the
pre-characterizations about inclined curves and slant helices and
restructure them with the necessary and sufficient condition [17].

Similiar to the working in \cite{zip}, in this work, we define $V_{n}-$
slant helix and give characterizations about the helix with necessary and
sufficient condition in $n-$ dimensional pseudo-Riemannian manifolds for the
first time.

\section{Preliminaries}

In this section, we give some basic definitions from differential geometry.

\begin{definition}
A metric tensor $g$ on a smooth manifold $M$ is a symmetric non-degenerate
(0,2) tensor field on $M$.

In other words, $g\left( X,Y\right) =g\left( Y,X\right) $ for all $X,Y\in TM$
(tangent bundle) and at the each point $p$ of $M$ if $g\left(
X_{p},Y_{p}\right) =0$ for all $Y_{p}\in T_{p}\left( M\right) $ , then $%
X_{p}=0$ (non-degenerate), where $T_{p}\left( M\right) $ is the tangent
space of $M$ at the point $p$ and $g:T_{p}\left( M\right) \times T_{p}\left(
M\right) \rightarrow 
\mathbb{R}
$ \cite{neill}.
\end{definition}

\begin{definition}
A pseudo-Riemannian manifold (or semi-Riemannian manifold) is a smooth
manifold $M$ furnished with a metric tensor $g$. That is, a
pseudo-Riemannian manifold is an ordered pair $\left( M,g\right) $ \cite%
{neill}.
\end{definition}

\begin{definition}
We shall recall the notion of a proper curve of order $n$ in a $n$%
-dimensional pseudo-Riemannian manifold $M$ with the metric tensor $g$. Let $%
\alpha :I\rightarrow M$ be a non-null curve in $M$ parametrized by the
arclength $s$, where $I$ is an open interval of the real line $%
\mathbb{R}
$. We denote the tangent vector field of $\alpha $ by $V_{1}$. We assume
that $\alpha $ satisfies the following Frenet formula: 
\begin{eqnarray*}
\nabla _{V_{1}}V_{1} &=&k_{1}V_{2}, \\
\nabla _{V_{1}}V_{i} &=&-\varepsilon _{i-2}\varepsilon
_{i-1}k_{i-1}V_{i-1}+k_{i}V_{i+1},\text{ }1<i<n \\
\nabla _{V_{1}}V_{n} &=&-\varepsilon _{n-2}\varepsilon _{n-1}k_{n-1}V_{n-1},
\end{eqnarray*}%
where%
\begin{eqnarray*}
k_{1} &=&\left \Vert \nabla _{V_{1}}V_{1}\right \Vert >0 \\
k_{i} &=&\left \Vert \nabla _{V_{1}}V_{i}+\varepsilon _{i-2}\varepsilon
_{i-1}k_{i-1}V_{i-1}\right \Vert >0,\text{ \ }2\leq i\leq n-1 \\
\varepsilon _{j-1} &=&g\left( V_{j},V_{j}\right) \text{ }\left( =\pm
1\right) \text{ },\text{ }1\leq j\leq n,\text{on }I\text{, }
\end{eqnarray*}%
and $\nabla $ is Levi-Civita connection of $M$.

We call such a curve a proper curve of order $n$, $k_{i}$ $\left( 1\leq
i\leq n-1\right) $ its $i-th$ curvature and $V_{1},...,V_{n}$ its Frenet
Frame field.

Morever, let us recall that $\left \Vert X\right \Vert =\sqrt{\left \vert
g\left( X,X\right) \right \vert }$ for $X\in TM$, where $TM$ is the tangent
bundle of $M$ \cite{song}.
\end{definition}

\section{$V_{n}$-slant helices and their harmonic curvature functions}

In this section, we give definition of a $V_{n}-$slant helix curve in a $n-$%
dimensional pseudo-Riemannian manifold. Furthermore, we give
characterizations by using harmonic curvatures for $V_{n}-$slant helices.

\begin{definition}
\textbf{\ }Let $M$ be a $n$-dimensional pseudo-Riemannian manifold and let $%
\alpha \left( s\right) $ be a proper curve of order $n$ (non-null) with the
curvatures $k_{i}$ $\left( i=1,...,n-1\right) $ in $M$. Then, harmonic
curvature functions of $\alpha $ are defined by 
\begin{equation*}
H_{i}^{\ast }:I\subset 
\mathbb{R}
\rightarrow 
\mathbb{R}%
\end{equation*}%
along $\alpha $ in $M$, where%
\begin{eqnarray*}
H_{0}^{\ast } &=&0, \\
H_{1}^{\ast } &=&\varepsilon _{n-3}\varepsilon _{n-2}\frac{k_{n-1}}{k_{n-2}},
\\
H_{i}^{\ast } &=&\left( k_{n-i}H_{i-2}^{\ast }-\nabla _{V_{1}}H_{i-1}^{\ast
}\right) \frac{\varepsilon _{n-\left( i+2\right) }\varepsilon _{n-\left(
i+1\right) }}{k_{n-\left( i+1\right) }},\text{ }2\leq i\leq n-2\text{.}
\end{eqnarray*}%
Note that $\nabla _{V_{1}}H_{i-1}^{\ast }=V_{1}\left( H_{i-1}^{\ast }\right)
=H_{i-1}^{\ast \prime }$.
\end{definition}

\begin{definition}
Let $\left( M,g\right) $ be a $n$-dimensional pseudo-Riemannian manifold and
let $\alpha \left( s\right) $ be a proper curve of order $n$ (non-null).We
call $\alpha $ as a $V_{n}$-slant helix in $M$ if the function%
\begin{equation*}
g\left( V_{n},X\right)
\end{equation*}%
is a non-zero constant along $\alpha $ and $X$ is a parallel vector field
along $\alpha $ in $M$ (i.e. $\nabla _{V_{1}}X=0$). Here, $V_{n}$ is $n$-th
Frenet Frame field and $X\in TM$. Also, $X$ is called the axis of $\alpha $.
\end{definition}

\begin{lemma}
Let $\left( M,g\right) $ be a $n$-dimensional pseudo-Riemannian manifold and
let $\alpha \left( s\right) $ be a proper curve of order $n$ (non-null). Let
us assume that $H_{n-2}^{\ast }\neq 0$ for $i=n-2$. Then, $\varepsilon
_{n-3}H_{1}^{\ast 2}+\varepsilon _{n-4}H_{2}^{\ast 2}+...+\varepsilon
_{0}H_{n-2}^{\ast 2}$ is non-zero constant if and only if $V_{1}\left(
H_{n-2}^{\ast }\right) =H_{n-2}^{\ast \prime }=k_{1}H_{n-3}^{\ast }$, where $%
V_{1}$ and $\left \{ H_{1}^{\ast },...,H_{n-2}^{\ast }\right \} $ are the
unit tangent vector field and the harmonic curvatures of $\alpha $,
respectively.
\end{lemma}

\begin{proof}
First, we assume that $\varepsilon _{n-3}H_{1}^{\ast 2}+\varepsilon
_{n-4}H_{2}^{\ast 2}+...+\varepsilon _{0}H_{n-2}^{\ast 2}$ is non-zero
constant. Consider the following functions given in Definition 3.1%
\begin{equation*}
H_{i}^{\ast }=\left( k_{n-i}H_{i-2}^{\ast }-H_{i-1}^{\ast \prime }\right) 
\frac{\varepsilon _{n-\left( i+2\right) }\varepsilon _{n-\left( i+1\right) }%
}{k_{n-\left( i+1\right) }}
\end{equation*}%
for $3\leq i\leq n-2$. So, from the equality, we can write%
\begin{equation}
k_{n-\left( i+1\right) }H_{i}^{\ast }=\varepsilon _{n-\left( i+2\right)
}\varepsilon _{n-\left( i+1\right) }\left( k_{n-i}H_{i-2}^{\ast
}-H_{i-1}^{\ast \prime }\right) \text{.}
\end{equation}%
Hence, in (3.1), if we take $i+1$ instead of $i$, we get%
\begin{equation}
\varepsilon _{n-\left( i+3\right) }\varepsilon _{n-\left( i+2\right)
}H_{i}^{\ast \prime }=\varepsilon _{n-\left( i+3\right) }\varepsilon
_{n-\left( i+2\right) }k_{n-\left( i+1\right) }H_{i-1}^{\ast }-k_{n-\left(
i+2\right) }H_{i+1}^{\ast },\text{ }2\leq i\leq n-3
\end{equation}%
together with%
\begin{equation*}
H_{1}^{\ast \prime }=-\frac{1}{\varepsilon _{n-4}\varepsilon _{n-3}}%
k_{n-3}H_{2}^{\ast }
\end{equation*}%
or 
\begin{equation}
H_{1}^{\ast \prime }=-\varepsilon _{n-4}\varepsilon _{n-3}k_{n-3}H_{2}^{\ast
}\text{.}
\end{equation}%
On the other hand, since $\varepsilon _{n-3}H_{1}^{\ast 2}+\varepsilon
_{n-4}H_{2}^{\ast 2}+...+\varepsilon _{0}H_{n-2}^{\ast 2}$ is constant, we
have%
\begin{equation*}
\varepsilon _{n-3}H_{1}^{\ast }H_{1}^{\ast \prime }+\varepsilon
_{n-4}H_{2}^{\ast }H_{2}^{\ast \prime }+...+\varepsilon _{0}H_{n-2}^{\ast
}H_{n-2}^{\ast \prime }=0
\end{equation*}%
and so, 
\begin{equation}
\varepsilon _{0}H_{n-2}^{\ast }H_{n-2}^{\ast \prime }=-\varepsilon
_{n-3}H_{1}^{\ast }H_{1}^{\ast \prime }-\varepsilon _{n-4}H_{2}^{\ast
}H_{2}^{\ast \prime }-...-\varepsilon _{1}H_{n-3}^{\ast }H_{n-3}^{\ast
\prime }\text{.}
\end{equation}%
By using (3.2) and (3.3), we obtain 
\begin{equation}
H_{1}^{\ast }H_{1}^{\ast \prime }=-\varepsilon _{n-4}\varepsilon
_{n-3}k_{n-3}H_{1}^{\ast }H_{2}^{\ast }
\end{equation}%
and%
\begin{equation}
\varepsilon _{n-\left( i+3\right) }\varepsilon _{n-\left( i+2\right)
}H_{i}^{\ast }H_{i}^{\ast \prime }=\varepsilon _{n-\left( i+3\right)
}\varepsilon _{n-\left( i+2\right) }k_{n-\left( i+1\right) }H_{i-1}^{\ast
}H_{i}^{\ast }-k_{n-\left( i+2\right) }H_{i}^{\ast }H_{i+1}^{\ast },\text{ }%
2\leq i\leq n-3\text{.}
\end{equation}%
Therefore, by using (3.4), (3.5) and (3.6), an algebraic calculus shows that 
\begin{equation*}
\varepsilon _{0}H_{n-2}^{\ast }H_{n-2}^{\ast \prime }=\varepsilon
_{0}k_{1}H_{n-3}^{\ast }H_{n-2}^{\ast }
\end{equation*}%
or%
\begin{equation*}
H_{n-2}^{\ast }H_{n-2}^{\ast \prime }=k_{1}H_{n-3}^{\ast }H_{n-2}^{\ast }%
\text{.}
\end{equation*}%
Since $H_{n-2}^{\ast }\neq 0$, we get the relation%
\begin{equation*}
H_{n-2}^{\ast \prime }=k_{1}H_{n-3}^{\ast }\text{.}
\end{equation*}

Conversely, we assume that%
\begin{equation}
H_{n-2}^{\ast \prime }=k_{1}H_{n-3}^{\ast }\text{.}
\end{equation}%
By using (3.7) and $H_{n-2}^{\ast }\neq 0$, we can write%
\begin{equation}
H_{n-2}^{\ast }H_{n-2}^{\ast \prime }=k_{1}H_{n-2}^{\ast }H_{n-3}^{\ast }%
\text{.}
\end{equation}%
From (3.6), we have the following equation sysytem:%
\begin{eqnarray*}
\text{for }i &=&n-3\text{, \  \  \  \  \ }\varepsilon _{1}H_{n-3}^{\ast
}H_{n-3}^{\ast \prime }=\varepsilon _{1}k_{2}H_{n-4}^{\ast }H_{n-3}^{\ast
}-\varepsilon _{0}k_{1}H_{n-3}^{\ast }H_{n-2}^{\ast }\text{,} \\
\text{for }i &=&n-4\text{, \  \  \  \  \ }\varepsilon _{2}H_{n-4}^{\ast
}H_{n-4}^{\ast \prime }=\varepsilon _{2}k_{3}H_{n-5}^{\ast }H_{n-4}^{\ast
}-\varepsilon _{1}k_{2}H_{n-4}^{\ast }H_{n-3}^{\ast }\text{,} \\
\text{for }i &=&n-5\text{, \  \  \  \  \ }\varepsilon _{3}H_{n-5}^{\ast
}H_{n-5}^{\ast \prime }=\varepsilon _{3}k_{4}H_{n-6}^{\ast }H_{n-5}^{\ast
}-\varepsilon _{2}k_{3}H_{n-5}^{\ast }H_{n-4}^{\ast }\text{,} \\
&&\cdot \\
&&\cdot \\
&&\cdot \\
\text{for }i &=&2\text{, \  \  \  \  \  \  \  \  \ }\varepsilon _{n-4}H_{2}^{\ast
}H_{2}^{\ast \prime }=\varepsilon _{n-4}k_{n-3}H_{1}^{\ast }H_{2}^{\ast
}-\varepsilon _{n-5}k_{n-4}H_{2}^{\ast }H_{3}^{\ast }\text{ .}
\end{eqnarray*}%
Moreover, from (3.5) and (3.8), we obtain%
\begin{equation}
\varepsilon _{n-3}H_{1}^{\ast }H_{1}^{\ast \prime }=-\varepsilon
_{n-4}k_{n-3}H_{1}^{\ast }H_{2}^{\ast }
\end{equation}%
and%
\begin{equation}
\varepsilon _{0}H_{n-2}^{\ast }H_{n-2}^{\ast \prime }=\varepsilon
_{0}k_{1}H_{n-2}^{\ast }H_{n-3}^{\ast }\text{.}
\end{equation}%
So, by using the above equation system and considering (3.9) and (3.10), an
algebraic calculus shows that%
\begin{equation}
\varepsilon _{n-3}H_{1}^{\ast }H_{1}^{\ast \prime }+\varepsilon
_{n-4}H_{2}^{\ast }H_{2}^{\ast \prime }+...+\varepsilon _{0}H_{n-2}^{\ast
}H_{n-2}^{\ast \prime }=0\text{.}
\end{equation}%
And, by integrating (3.11), we can easily say that%
\begin{equation*}
\varepsilon _{n-3}H_{1}^{\ast 2}+\varepsilon _{n-4}H_{2}^{\ast
2}+...+\varepsilon _{0}H_{n-2}^{\ast 2}
\end{equation*}%
is a non-zero constant. This completes the proof.
\end{proof}

\noindent \textbf{Proposition 3.1. }Let $\left( M,g\right) $ be a $n$%
-dimensional pseudo-Riemannian manifold and let $\alpha \left( s\right) $ be
a proper curve of order $n$ (non-null). If $\alpha $ is an $V_{n}$ -slant
helix in $M$, then we have%
\begin{equation}
g\left( V_{n-\left( i+1\right) },X\right) =H_{i}^{\ast }g\left(
V_{n},X\right) ,\text{ }i=0,1,...,n-2\text{,}
\end{equation}%
where $X$ is the axis of $\alpha $. Here, $\left \{
V_{1},V_{2},...,V_{n}\right \} $ denote the Frenet Frame of $\alpha $ and $%
\left \{ H_{1}^{\ast },H_{2}^{\ast },...,H_{n-2}^{\ast }\right \} $ denote
the harmonic curvature functions of the curve $\alpha $.

\begin{proof}
We will use the induction method. Let $i=1$. Since $X$ is the axis of the $%
V_{n}$- slant helix $\alpha $, we get%
\begin{equation*}
X=\lambda _{1}V_{1}+...+\lambda _{n}V_{n}\text{.}
\end{equation*}%
From the definition of $\ V_{n}$ -slant helix, we have%
\begin{equation}
g\left( V_{n},X\right) =\lambda _{n}\varepsilon _{n-1}=\text{constant.}
\end{equation}%
A differentiation in Eq. (3.13) and the Frenet formulas gives us that%
\begin{equation}
g\left( V_{n-1},X\right) =0\text{.}
\end{equation}%
Again, differentiation in Eq. (3.14) and the Frenet formulas give%
\begin{eqnarray*}
g\left( \nabla _{V_{1}}V_{n-1},X\right) &=&0\text{,} \\
-\varepsilon _{n-3}\varepsilon _{n-2}k_{n-2}g\left( V_{n-2},X\right)
+k_{n-1}g\left( V_{n},X\right) &=&0\text{,} \\
g\left( V_{n-2},X\right) &=&\varepsilon _{n-3}\varepsilon _{n-2}\frac{k_{n-1}%
}{k_{n-2}}g\left( V_{n},X\right) \text{,} \\
g\left( V_{n-2},X\right) &=&H_{1}^{\ast }g\left( V_{n},X\right) \text{,}
\end{eqnarray*}%
respectively. Hence, it is shown that the Eq. (3.12) is true for $i=1$.

We now assume the Eq. (3.12) is true for the first $i-1$. Then, we have%
\begin{equation}
g\left( V_{n-i},X\right) =H_{i-1}^{\ast }g\left( V_{n},X\right) \text{.}
\end{equation}%
A differentiation in Eq. (3.15) and the Frenet formulas give us that%
\begin{equation*}
-\varepsilon _{n-i-2}\varepsilon _{n-i-1}k_{n-i-1}g\left( V_{n-i-1},X\right)
+k_{n-i}g\left( V_{n-i+1},X\right) =\nabla _{V_{1}}H_{i-1}^{\ast }g\left(
V_{n},X\right) \text{.}
\end{equation*}%
Since we have the induction hypothesis, $g\left( V_{n-i+1},X\right)
=H_{i-2}^{\ast }g\left( V_{n},X\right) $, we get 
\begin{equation*}
\left( k_{n-i}H_{i-2}^{\ast }-\nabla _{V_{1}}H_{i-1}^{\ast }\right) \frac{%
\varepsilon _{n-\left( i+2\right) }\varepsilon _{n-\left( i+1\right) }}{%
k_{n-\left( i+1\right) }}g\left( V_{n},X\right) =g\left( V_{n-\left(
i+1\right) },X\right) \text{,}
\end{equation*}%
which gives%
\begin{equation*}
g\left( V_{n-\left( i+1\right) },X\right) =H_{i}^{\ast }g\left(
V_{n},X\right) \text{.}
\end{equation*}
\end{proof}

\begin{theorem}
Let $\left( M,g\right) $ be a $n$-dimensional pseudo-Riemannian manifold and
let $\alpha \left( s\right) $ be a proper curve of order $n$ (non-null).
Then, $\alpha $ is a $V_{n}$-slant helix in $M$ if and only if it satisfies
that%
\begin{equation*}
\sum \limits_{i=1}^{n-2}\varepsilon _{n-\left( i+2\right) }H_{i}^{\ast 2}
\end{equation*}%
is equal to non-zero constant and $H_{n-2}^{\ast }\neq 0$.
\end{theorem}

\begin{proof}
Suppose $\alpha $ be a $V_{n}$ -slant helix. According to the Definition 3.2
and the proof of \ Proposition 3.1, 
\begin{equation}
g\left( V_{n},X\right) =\lambda _{n}\varepsilon _{n-1}=\text{constant,}
\end{equation}%
where $X$ the axis of $\alpha $. From Proposition 3.1., we have 
\begin{equation}
g\left( V_{n-\left( i+1\right) },X\right) =H_{i}^{\ast }g\left(
V_{n},X\right)
\end{equation}%
for $1\leq i\leq n-2$. Moreover, from (3.16) and Frenet formulas, we can
write%
\begin{equation*}
-\varepsilon _{n-2}\varepsilon _{n-1}k_{n-1}g\left( V_{n-1},X\right) =0\text{%
.}
\end{equation*}%
Since $-\varepsilon _{n-2}\varepsilon _{n-1}k_{n-1}$ is different from zero, 
$g\left( V_{n-1},X\right) =0$. It is known that the system $\left \{
V_{1},...,V_{n}\right \} $ is a basis of $\varkappa \left( M\right) $
(tangent bundle) along $\alpha $. Hence, $X$ can be expressed in the form%
\begin{equation}
X=\dsum \limits_{i=1}^{n}\lambda _{i}V_{i}\text{.}
\end{equation}%
Moreover, from (3.18), we get the system%
\begin{eqnarray*}
\varepsilon _{0}\lambda _{1} &=&g\left( X,V_{1}\right) \\
\varepsilon _{1}\lambda _{2} &=&g\left( X,V_{2}\right) \\
&&. \\
&&. \\
&&. \\
\varepsilon _{n-3}\lambda _{n-2} &=&g\left( X,V_{n-2}\right) \\
\varepsilon _{n-2}\lambda _{n-1} &=&g\left( X,V_{n-1}\right) =0 \\
\varepsilon _{n-1}\lambda _{n} &=&g\left( X,V_{n}\right)
\end{eqnarray*}%
by using the metric $g$. Therefore, from Proposition 3.1 and the above
system, we can see that the following system is true: 
\begin{eqnarray*}
\lambda _{1} &=&g\left( X,V_{1}\right) =\varepsilon _{0}H_{n-2}^{\ast
}g\left( X,V_{n}\right) \\
\lambda _{2} &=&g\left( X,V_{2}\right) =\varepsilon _{1}H_{n-3}^{\ast
}g\left( X,V_{n}\right) \\
&&. \\
&&. \\
&&. \\
\lambda _{n-2} &=&g\left( X,V_{n-2}\right) =\varepsilon _{n-3}H_{1}^{\ast
}g\left( X,V_{n}\right) \\
\lambda _{n-1} &=&g\left( X,V_{n-1}\right) =0 \\
\lambda _{n} &=&\varepsilon _{n-1}g\left( X,V_{n}\right) \text{.}
\end{eqnarray*}%
Thus, it can be easily obtained the axis of the curve $\alpha $ as%
\begin{equation}
X=g\left( X,V_{n}\right) \left \{ \dsum \limits_{i=1}^{n-2}H_{i}^{\ast
}V_{n-(i+1)}\varepsilon _{n-(i+2)}+\left( \varepsilon _{n-1}V_{n}\right)
\right \}
\end{equation}%
by making use of the equality (3.18) and the last system.

Therefore, from (3.19), we can write%
\begin{equation}
g\left( X,X\right) =\left[ g\left( X,V_{n}\right) \right] ^{2}\left(
\varepsilon _{0}^{3}H_{n-2}^{\ast 2}+...+\varepsilon _{n-3}^{3}H_{1}^{\ast
2}+\varepsilon _{n-1}^{3}\right) \text{.}
\end{equation}%
Morever, by the definition of metric tensor, we have%
\begin{equation*}
\left \vert g\left( X,X\right) \right \vert =\left \Vert X\right \Vert ^{2}%
\text{.}
\end{equation*}%
Since $\alpha $ is a $V_{n}$-slant helix, $\left \Vert X\right \Vert =$%
constant and $g\left( X,V_{n}\right) $ is non-zero constant along $\alpha $.
Hence, from (3.20), we obtain%
\begin{equation*}
\varepsilon _{0}^{3}H_{n-2}^{\ast 2}+...+\varepsilon _{n-3}^{3}H_{1}^{\ast
2}+\varepsilon _{n-1}^{3}
\end{equation*}%
is constant. In other words,%
\begin{equation*}
\varepsilon _{0}H_{n-2}^{\ast 2}+...+\varepsilon _{n-3}H_{1}^{\ast 2}=\dsum
\limits_{i=1}^{n-2}\varepsilon _{n-(i+2)}H_{i}^{\ast 2}
\end{equation*}%
is constant.

Now, we will show that $H_{n-2}^{\ast }\neq 0$. We assume that $%
H_{n-2}^{\ast }=0$. Then, for $i=n-2$ in (3.17), we have%
\begin{equation}
g\left( V_{1},X\right) =H_{n-2}^{\ast }g\left( X,V_{n}\right) =0\text{.}
\end{equation}%
If we take derivative in each part of (3.21) in the direction $V_{1}$ on $M$%
, then we have%
\begin{equation}
g\left( \nabla _{V_{1}}V_{1},X\right) +g\left( V_{1},\nabla _{V_{1}}X\right)
=0\text{.}
\end{equation}%
On the other hand, $\nabla _{V_{1}}X$ $=0$ since $\alpha $ is a $V_{n}-$%
slant helix. Then, from (3.22), we have%
\begin{equation*}
g\left( \nabla _{V_{1}}V_{1},X\right) =k_{1}g\left( V_{2},X\right) =0
\end{equation*}%
by using the Frenet formulas. Since $k_{1}$ is positive, it must be $g\left(
V_{2},X\right) =0$. Now, for $i=n-3$ in (3.17), 
\begin{equation*}
g\left( V_{2},X\right) =H_{n-3}^{\ast }\text{ }g\left( V_{n},X\right) \text{.%
}
\end{equation*}%
Since $g\left( V_{2},X\right) =0$ and $g\left( V_{n},X\right) \neq 0$, it
must be $H_{n-3}^{\ast }=0$. Continuing this process, we get $H_{1}^{\ast }=0
$. Let us recall that $H_{1}^{\ast }=\varepsilon _{n-3}\varepsilon _{n-2}%
\dfrac{k_{n-1}}{k_{n-2}}$, thus we have a contradiction because all the
curvatures are nowhere zero. Consequently, $H_{n-2}^{\ast }\neq 0$.

Conversely, we assume that $\dsum \limits_{i=1}^{n-2}\varepsilon
_{n-(i+2)}H_{i}^{\ast 2}=$constant and $H_{n-2}^{\ast }\neq 0$. We take the
vector field%
\begin{equation*}
X=\lambda _{n}V_{n}+\dsum \limits_{i=1}^{n-2}\lambda _{n}\varepsilon
_{n-1}\varepsilon _{n-\left( i+2\right) }H_{i}^{\ast }V_{n-\left( i+1\right)
}
\end{equation*}%
or%
\begin{equation*}
X=\lambda _{n}V_{n}+\lambda _{n}\varepsilon _{n-1}\dsum
\limits_{i=3}^{n}\varepsilon _{n-i}H_{i-2}^{\ast }V_{n-\left( i-1\right) }%
\text{,}
\end{equation*}%
where $\lambda _{n}$ is constant. We will show that it is parallel along $%
\alpha $, i.e. $\nabla _{V_{1}}X=0$. By direct calculation, we have%
\begin{eqnarray*}
\nabla _{V_{1}}X &=&\nabla _{V_{1}}\left( \lambda _{n}V_{n}\right) +\lambda
_{n}\varepsilon _{n-1}\dsum \limits_{i=3}^{n}\varepsilon _{n-i}\nabla
_{V_{1}}\left( H_{i-2}^{\ast }V_{n-\left( i-1\right) }\right) \\
&=&\lambda _{n}\nabla _{V_{1}}V_{n}+\lambda _{n}\varepsilon _{n-1}\dsum
\limits_{i=3}^{n}\varepsilon _{n-i}\left[ H_{i-2}^{\ast \prime }V_{n-\left(
i-1\right) }+H_{i-2}^{\ast }\nabla _{V_{1}}\left( V_{n-\left( i-1\right)
}\right) \right] \\
&=&\lambda _{n}\varepsilon _{n-1}[-\varepsilon _{n-2}k_{n-1}V_{n-1}+(\dsum
\limits_{i=3}^{n-1}\varepsilon _{n-i}H_{i-2}^{\ast \prime }V_{n-\left(
i-1\right) }-\varepsilon _{n-\left( i+1\right) }k_{n-i}V_{n-i}H_{i-2}^{\ast
}+ \\
&&k_{n-\left( i-1\right) }V_{n-\left( i-2\right) }H_{i-2}^{\ast }\varepsilon
_{n-i})+\varepsilon _{0}H_{n-2}^{\ast \prime }V_{1}+\varepsilon
_{0}k_{1}H_{n-2}^{\ast }V_{2}]\text{.}
\end{eqnarray*}%
Here, in the case $n=3$, we omit the term of sum.

On the other hand, by using (3.2), we can write%
\begin{equation}
\varepsilon _{n-\left( i+1\right) }\varepsilon _{n-i}H_{i-2}^{\ast \prime
}=\varepsilon _{n-\left( i+1\right) }\varepsilon _{n-i}k_{n-\left(
i-1\right) }H_{i-3}^{\ast }-k_{n-i}H_{i-1}^{\ast }
\end{equation}%
for $4\leq i\leq n-1$ together with (3.3). Moreover, from Lemma 3.1, we know
that%
\begin{equation}
H_{n-2}^{\ast \prime }=k_{1}H_{n-3}^{\ast }\text{.}
\end{equation}%
Therefore, by using (3.3), (3.23), (3.24) and by the definition of $%
H_{1}^{\ast }$, algebraic calculus shows that $\nabla _{V_{1}}X=0$. Besides, 
$g\left( V_{n},X\right) =\lambda _{n}\varepsilon _{n-1}$ is constant.
Consequently, $\alpha $ is a $V_{n}$-slant helix in $M$.
\end{proof}

\begin{corollary}
Let $\left( M,g\right) $ be a $n$-dimensional pseudo-Riemannian manifold and
let $\alpha \left( s\right) $ be a proper curve of order $n$ (non-null).
Then, $\alpha $ is a $V_{n}$-slant helix in $M$ if and only if $%
H_{n-2}^{\ast \prime }=k_{1}H_{n-3}^{\ast }$ and $H_{n-2}^{\ast }\neq 0$,
where $\left \{ H_{1}^{\ast },H_{2}^{\ast },...,H_{n-2}^{\ast }\right \} $
denote the harmonic curvature functions of the curve $\alpha $.
\end{corollary}

\begin{proof}
It is obvious by using Lemma 3.1. and Theorem 3.1.
\end{proof}

\end{document}